\documentclass{article}
\usepackage{amsmath}
\usepackage{amsthm}
\usepackage{amssymb}
\usepackage{wasysym}
\usepackage{mathrsfs}
\usepackage{bbm, bm}
\usepackage{geometry} 
\usepackage{mathtools}
\usepackage[hang]{footmisc} 
\usepackage{color}
\usepackage{verbatim}

\usepackage{pgf,tikz,pgfplots,tikz-3dplot-circleofsphere,tkz-euclide}
\pgfplotsset{compat=1.14}
\usetikzlibrary{arrows}

\DeclareMathOperator{\R}{\mathbb{R}}

\DeclareMathOperator{\affine}{aff}

\DeclareMathOperator{\dist}{dist}

\DeclareMathOperator{\relint}{relint}
\DeclareMathOperator{\intt}{int}

\newcommand{\Sp}{\mathbb{S}^{d-1}}
\newcommand{\Stwo}{\mathbb{S}^2}

\let\originalleft\left
\let\originalright\right
\renewcommand{\left}{\mathopen{}\mathclose\bgroup\originalleft}
\renewcommand{\right}{\aftergroup\egroup\originalright}

\newtheorem{theo}{Theorem}

\newtheorem{cor}{Corollary}
\newtheorem{lem}{Lemma}
\theoremstyle{definition}
\newtheorem{remark}{Remark}

\title{The Maximum Surface Area Polyhedron with Five Vertices Inscribed in the Sphere $\mathbb{S}^2$}

\author{Jessica Donahue, Steven Hoehner and Ben Li}

\begin{document}
\setcounter{footnote}{0}
\maketitle
\begin{abstract}
This article focuses on the problem of analytically determining the optimal placement of  five points on the unit sphere $\mathbb{S}^2$ so that the surface area of the convex hull of the points is maximized. It is  shown that the optimal polyhedron has a trigonal bipyramidal structure with two vertices  placed at the north and south poles and the other three vertices forming an equilateral triangle inscribed in the equator. This result confirms a conjecture of Akkiraju, who conducted a numerical search for the maximizer. As an application to crystallography,  the surface area discrepancy is considered as a measure of distortion between an observed coordination polyhedron and an ideal one. The main result yields a formula for the surface area discrepancy of any coordination polyhedron with five vertices.
\end{abstract}

\renewcommand{\thefootnote}{}

\renewcommand{\thefootnote}{\arabic{footnote}}
\setcounter{footnote}{0}




\section{Introduction and Main Result}

The coordination polyhedron of a configuration of  ligand atoms on the unit sphere $\Stwo$ bonded to a central atom at the origin is a fundamental concept in crystallography \cite{Borchardt-Ott}. It is a natural question to compare the shape of an observed structure with that of a regular or ``ideal" polyhedron, since usually the distortion turns out to be attributed to the distribution of cations. Various notions of ``idealness" and  measures of distortion have been studied in crystallography. Examples include differences in bond lengths, local symmetry and bond angle strains  \cite{Robinson1971}.  The {\it volume discrepancy} $\mathscr{V}(P)$ of an observed coordination polyhedron $P$ was introduced in \cite{Makovicky} as a way to measure the distortion of  $P$  from an  ``ideal" polyhedron $\widehat{Q}_{\text{vol}}(P)$  inscribed in $\Stwo$ that  is combinatorially equivalent (see Section 3 for the definition) to $P$ and achieves the maximum volume. It was defined  as
\begin{equation}\label{volumedistortion}
   \mathscr{V}(P):= \frac{V(\widehat{Q}_{\text{vol}}(P))-V(P)}{V(\widehat{Q}_{\text{vol}}(P))}
\end{equation}
where $V(P)$ denotes the volume of $P$. The existence of the ideal polyhedron $\widehat{Q}_{\text{vol}}(P)$ follows from a compactness argument. 

The volume discrepancy has found a number of applications to crystallography. In \cite{Makovicky}, the authors analyzed several  structural families using the volume discrepancy functional. It turned out that  \eqref{volumedistortion} can be used as a global measure itself or combined with other distortion characteristics to quantify, for instance, departures from a structural archetype \cite{Makovicky1989} and the configurational driving mechanisms for phase transformations \cite{Makovicky}. 
   
In view of \eqref{volumedistortion}, it is natural to instead define idealness and measure distortion in terms of surface area. Given a coordination polyhedron $P$, let $\widehat{Q}_{\text{SA}}(P)$ denote a combinatorially equivalent polyhedron inscribed in $\Stwo$ that achieves the maximum surface area. It follows again by a compactness argument that the ideal polyhedron $\widehat{Q}_{\text{SA}}(P)$ exists. Thus,  one may define the {\it surface area discrepancy} $\mathscr{S}(P)$ of the coordination polyhedron $P$ by
\begin{equation}\label{SAdiscrepancy}
   \mathscr{S}(P):= \frac{S(\widehat{Q}_{\text{SA}}(P))-S(P)}{S(\widehat{Q}_{\text{SA}}(P))}
\end{equation}
where $S(P)$ denotes the surface area of $P$. 

The surface area maximizers among {\it all} polyhedra with $v$ vertices inscribed in $\Stwo$ have been determined analytically for the cases $v=4,6$ and $12$ \cite{Heppes, Krammer, Toth}; see the discussion in Section \ref{relatedresults}. These three cases are exceptional because there exists a regular polyhedron whose facets are congruent equilateral triangles. Using this property and a moment of inertia formula (e.g., \cite{BalkBoltyanskii}), it follows that the global surface area maximizers for $v=4,6,12$ are the regular tetrahedron, octahedron and icosahedron, respectively.  In the case $v=5$, however, there is no regular polyhedron  inscribed in the sphere, and the aforementioned moment of inertia formula yields a strict inequality.  Nevertheless, one may expect the optimizer to be a triangular bipyramid since this structure exhibits the highest degree of symmetry among all five point configurations on the sphere. Akkiraju \cite{akkiraju} conducted a numerical search for the global maximizer with five vertices, and asked for a proof that it is the triangular bipyramid  with two vertices at the north and south poles and three more forming an equilateral triangle in the equator.  To the best of our knowledge, however, a proof was missing until now. In our main result we close this gap and provide an affirmative answer to Akkiraju's question. 
To state the result, let $\bm{e}_1,\bm{e}_2,\bm{e}_3$ denote the standard basis vectors of $\R^3$. An illustration of the global surface area maximizer is shown in Figure 1 below.

\begin{center}
\tdplotsetmaincoords{80}{90}
\def\r{1}
  \begin{tikzpicture}[scale=3,line join=bevel, tdplot_main_coords]
    \coordinate (O) at (0,0,0);

\coordinate (A) at (1,0,0);
\coordinate (B) at ({-1/2},{sqrt(3)/2},0);
\coordinate (C) at ({-1/2},{-sqrt(3)/2},0);
\coordinate (D) at (0,0,1);
\coordinate (E) at (0,0,{-1});

\begin{scope}[thick]
    \draw (A) -- (D)--(B);
    \draw (A) -- (D) -- (C);
    \draw (A) -- (B)--(E);
    \draw (A)--(C)--(E);
    \draw (A)--(E)--(B);
\end{scope}

\draw[thick,fill=green, opacity=0.2] (A) -- (D)--(B);
\draw[thick,fill=green, opacity=0.2] (A) -- (D) -- (C);
\draw[thick,fill=green,opacity=0.2](A) -- (C) -- (E);  
\draw[thick,fill=green,opacity=0.2] (A)--(E)--(B);  

\begin{scope}[dashed] 
    \draw (C) -- (B);
\end{scope}

\begin{scope}[opacity=0.8]
\draw[tdplot_screen_coords] (0,0,0) circle (\r);
\tdplotCsDrawLatCircle{\r}{0}
\end{scope} 

\filldraw[black] (0,0,0) circle (0.25pt) node[anchor=east] {$\mathbf{0}$};
\filldraw[black] (0,0,1) circle (0.25pt) node[anchor=south] {$\bm{e}_3$};
\filldraw[black] (0,0,-1) circle (0.25pt) node[anchor=north] {$-\bm{e}_3$};
\filldraw[black] (1,0,0) circle (0.25pt) node[anchor=north west] {$\bm{e}_1$};
\filldraw[black] (B) circle (0.25pt) node[anchor=south] {$\bm{\zeta}_1$};
\filldraw[black] (C) circle (0.25pt) node[anchor=south] {$\bm{\zeta}_2$};

  \end{tikzpicture}
  
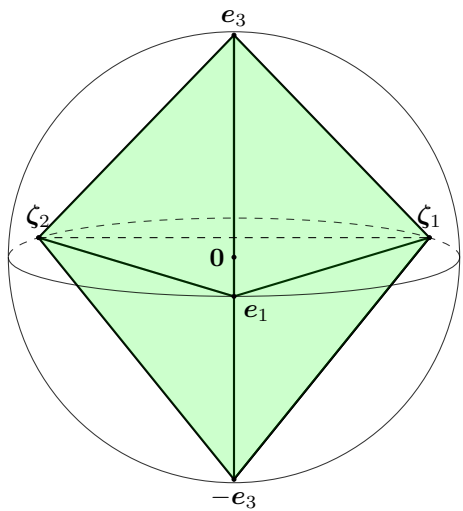
  {\footnotesize {\bf Figure 1}: The maximum surface area polyhedron with 5 vertices is the convex hull of the north and south poles $\pm \bm{e}_3$ and an equilateral triangle inscribed in the equator with vertices $\bm{e}_1,\bm{\zeta}_1$ and $\bm{\zeta}_2$.}
    \end{center}

\begin{theo}\label{mainThm}
Let $P$ be the convex hull of five points chosen from the unit sphere $\Stwo$. Then
\begin{equation}
S(P) \leq \frac{3\sqrt{15}}{2}=5.809475\ldots
\end{equation}
with equality if and only if $P$ is a rotation of the  triangular bipyramid with vertices $\bm{e}_3$, $-\bm{e}_3$, $\bm{e}_1$, $\bm{\zeta}_1:=(-\frac{1}{2}, \frac{\sqrt{3}}{2},0)$ and $\bm{\zeta}_2:=(-\frac{1}{2}, -\frac{\sqrt{3}}{2},0)$. 
\end{theo}
\noindent 

It follows that if $P$ is an observed polyhedron that is combinatorially equivalent to a triangular bipyramid, then
\begin{equation}\label{SAdiscform}
\mathscr{S}(P) = 1-\left(\frac{2}{3\sqrt{15}}\right)S(P).
\end{equation} 
In Section \ref{proofsection}, we derive a similar formula for the surface area discrepancy of an observed coordination polyhedron that is combinatorially equivalent to a square pyramidal structure. We then use this result to show that the volume and surface area discrepancies are not always equivalent. 

The problem of maximizing the volume or surface area of inscribed polyhedra also has applications to quantum theory. In this setting, every polyhedron inscribed in \(\Stwo\) with \(v\) vertices serves as a unique geometric representation of a pure symmetric state of a \(v\)-qubit system (see, e.g., \cite{Majorep, Kazakov}). For instance, the GHZ type state on a \(v\)-qubit  system corresponds to a (planar) regular $v$-polygon inscribed in \(\Stwo\). Such a polyhedron is called a {\it Majorana polyhedron (representation)}. For a pure symmetric state, the surface area of its Majorana polyhedron was proposed to  be a new measure of its entanglement \cite{Kazakov}. It was also conjectured  that this new measure is equivalent to the entropy measure of entanglement \cite{Kazakov}. This is certainly true for tripartite qubit-systems. Assuming the conjecture holds, our main result leads to an explicit five-partite state that bears maximal entanglement.

The rest of the paper is outlined as follows. In Section \ref{relatedresults} we describe the known results that are most closely related to Theorem \ref{mainThm}. In Section \ref{notations} we state the definitions and notations used throughout the paper. The proof of Theorem \ref{mainThm} is given in Section \ref{proofsection}. Next, in Section \ref{comparison} we compare the volume and surface area discrepancies for the classes of $v$-pyramids and $v$-bipyramids, showing they are distinct in the former and equivalent in the latter. Finally, in Section \ref{Problems} we summarize our results and  discuss some related open problems.

 \section{Related Results}\label{relatedresults}

We will now briefly discuss the known results in convex and discrete geometry that are most closely related to Theorem \ref{mainThm}. Let  $P$ be  a polyhedron in $\R^3$ with $v$ vertices, $e$ edges and $f$ facets. Assume that the vertices of $P$ lie in the unit sphere $\Stwo$, and that $P$ satisfies the ``foot condition" in which the feet of the perpendiculars  from the circumcenter (or incenter, respectively) of $P$ to its facet-planes and edge-lines lie on the corresponding facets and edges. 
By a  result of L. Fejes T\'oth \cite{Toth} (see also \cite{Toth1950}, Theorem 2, p. 279), the surface area of such a polyhedron is bounded by
\begin{equation}\label{tothleq}
S(P) \leq e\sin\frac{\pi f}{e}\left(1-\cot^2\frac{\pi f}{2e}\cot^2\frac{\pi v}{2e}\right).
\end{equation}
Equality holds only for the regular polyhedra. Linhart \cite{Linhart} later proved that the condition on the facets is superfluous. It follows from \eqref{tothleq} that if $P$ is a polyhedron  with $v$ vertices inscribed in $\Stwo$ that satisfies the foot condition, then
\begin{equation}\label{florianleq}
    S(P) \leq \frac{3\sqrt{3}}{2}(v-2)\left(1-\frac{1}{3}\cot^2\frac{\pi v}{6(v-2)}\right)=:G(v)
\end{equation}
with equality if and only if $v=4,6$ or $12$ and $P$ is the regular tetrahedron, regular octahedron or regular icosahedron, respectively (see, e.g., Remark 1 in \cite{HorvathIcosahedron}). This result is remarkable as there are precisely  6,384,634 combinatorial types of polyhedra with 12 vertices (see \cite{DF1981} and the references therein)! 

To apply \eqref{florianleq}, we must show that any surface area maximizer satisfies the foot condition for the edges. Suppose that $Q_v^*$ is the polyhedron that achieves the maximum surface area among all polyhedra with at most $v$ vertices inscribed in the Euclidean unit ball $\mathbb{B}_2^3=\{\bm{x}\in\R^3: \|\bm{x}\|_2\leq 1\}$, where $\|\bm{x}\|_2=\sqrt{x_1^2+x_2^2+x_3^2}$ denotes the Euclidean norm of $\bm{x}=(x_1,x_2,x_3)\in\R^3$. Then  $Q_v^*$ has exactly $v$ vertices, and all vertices of $Q_v^*$ lie in the boundary of $\mathbb{B}_2^3$, namely, the unit sphere $\Stwo=\{\bm{x}\in\R^3: \|\bm{x}\|_2=1\}$. By continuity, for any edge $E$ of $P$ there exists a point $\bm{x}^*$ in the affine hull $\affine(E)$ of $E$ such that $\|\bm{x}^*\|_2$ is minimal. Since all vertices of $Q_v^*$ lie in the boundary we have $\affine(E)\cap\mathbb{B}_2^3=E$, which contains $\bm{x}^*$. This shows that the surface area maximizer satisfies the foot condition.  It now follows from \eqref{florianleq} that the regular tetrahedron, octahedron and icosahedron are the unique  surface area maximizers (up to rotations) for $v=4,6$ and 12, respectively. Prior to the work  \cite{Toth}, the case $v=4$ of the tetrahedron was  settled contemporaneously in \cite{Heppes} and  \cite{Krammer}. More generally, a  result in \cite{tanner} implies that among all simplices inscribed in the unit sphere $\mathbb{S}^{d-1}$ in $\R^d$, $d\geq 2$, the regular simplex has maximal surface area (a {\it simplex} in $\R^d$ is the convex hull of $d+1$ affinely independent points).  For more background on the known results in this area, we refer the reader to, e.g., \cite{HorvathLangi, MPVbook}.

We now turn our attention to the problem considered in this article. By \eqref{florianleq},
\[
G(5) = \frac{9\sqrt{3}}{2}\left(1-\frac{1}{3}\cot^2\frac{5\pi}{18}\right)\approx 5.96495.
\]
On the other hand, the  triangular bipyramid in Theorem \ref{mainThm} satisfies the foot condition and has surface area
\[\frac{3\sqrt{15}}{2}\approx 5.809<G(5).\] 
Thus, a strict inequality holds in \eqref{florianleq} for the triangular bipyramid  in Theorem \ref{mainThm}, so we need a new argument to prove that it  is the maximizer.  The main step is to determine the $v$-pyramid (definition below) of maximum lateral surface area contained in a cap of the sphere, assuming the base of the pyramid lies in the base of the cap. Theorem \ref{mainThm} then follows by applying this result to each  combinatorial type of polyhedron with $v=5$ vertices.

In the following table, we summarize the aforementioned cases where the global surface area maximizers have been explicitly determined.  For fixed $v\geq 4$, recall that $Q_v^*$ denotes the maximum surface area polyhedron with  $v$ vertices inscribed in $\Stwo$. In the last column, we give the surface area discrepancy $\mathscr{S}(P)$ of an observed coordination polyhedron $P$ that is combinatorially equivalent to the maximizer. Please note that a  table of $G(v)$ values can also be found in \cite{akkiraju}.   

\vspace{3mm}

\begin{table}
\caption{List of maximum surface area polyhedra with $v\leq 12$ vertices inscribed in $\Stwo$.}
\begin{center}
    \begin{tabular}{ccccc}
    $v$  & $G(v)$ & $S(Q_v^*)$ & Maximizer $Q_v^*$ & $\mathscr{S}(P)$\\ \hline
    4  & $8/\sqrt{3}\approx 4.62$ & $8/\sqrt{3}$ & regular tetrahedron & $1-\tfrac{\sqrt{3}}{8}S(P)$\\ \hline
    5   & $\approx 5.96$ & $3\sqrt{15}/2$ & triangular bipyramid & $1-\tfrac{2}{3\sqrt{15}}S(P)$\\ \hline
    6     & $4\sqrt{3}\approx 6.93$ & $4\sqrt{3}$ & regular octahedron & $1-\tfrac{1}{4\sqrt{3}}S(P)$\\ \hline
    7   & $\approx 7.65$ & -- & -- & --\\ \hline
    8  &  $\approx 8.21$ & -- & -- & --\\\hline
    9  & $\approx 8.65$ & -- & -- & --\\ \hline
    10 & $\approx 9.02$ & -- & -- & --\\ \hline
    11  &  $\approx 9.32$ & -- & -- & --\\ \hline
    12   & $2\sqrt{75}-2\sqrt{15}\approx 9.57$ & $2\sqrt{75}-2\sqrt{15}$ & regular icosahedron & $1-\tfrac{S(P)}{2\sqrt{75}-2\sqrt{15}}$
    \end{tabular}
    \end{center}
\end{table}

\section{Definitions and Notations}\label{notations}

We shall work in three-dimensional space $\R^3$ with standard basis $\{\bm{e}_1,\bm{e}_2,\bm{e}_3\}$ and inner product   $\langle\bm{x},\bm{y}\rangle=\sum_{i=1}^3 x_i y_i$, where $\bm{x},\bm{y}\in\R^3$. The Euclidean norm of $\bm{x}\in\R^3$ is $\|\bm{x}\|_2=\sqrt{\langle\bm{x},\bm{x}\rangle}=\sqrt{\sum_{i=1}^3 x_i^2}$. The Euclidean unit ball in $\R^3$ centered at the origin $\mathbf{0}$ is denoted by $\mathbb{B}_2^3=\{\bm{x}\in\R^3: \|\bm{x}\|_2\leq 1\}$. Its boundary is the unit sphere $\Stwo=\{\bm{x}\in\R^3: \|\bm{x}\|_2=1\}$. The distance from a point $\bm{x}\in\R^3$ to a closed set $A\subset\R^3$ is  $\dist(\bm{x},A)=\min_{\bm{y}\in A}\|\bm{x}-\bm{y}\|_2$. 

For a two-dimensional plane $H=H(\bm{u},b)=\{\bm{x}\in\R^3: \langle \bm{x},\bm{u}\rangle = b\}$ in $\R^3$ (where $\bm{u}\in\mathbb{S}^2$ and $b\in\R$),  we let $H^+=\{\bm{x}\in\R^3: \langle \bm{x},\bm{u}\rangle \geq b\}$ and $H^-=\{\bm{x}\in\R^3: \langle \bm{x},\bm{u}\rangle \leq b\}$ denote its two closed halfspaces (a plane in $\R^2$ and its two closed halfspaces  are defined analogously). The orthogonal complement $\bm{x}^\perp$ of a vector $\bm{x}\in\R^3$ is the set $\bm{x}^\perp=\{\bm{y}\in\R^3: \langle \bm{x},\bm{y}\rangle=0\}$.  The affine hull and the interior of a set $A\subset\R^3$ are denoted by $\affine(A)$ and $\intt(A)$, respectively. The interior of $A$ within its affine hull  is  called the {\it relative interior} of $A$ and is denoted by $\relint(A)$. 

A set $C\subset\R^3$ is {\it convex} if for any $\bm{x},\bm{y}\in C$, the line segment $[\bm{x},\bm{y}]$ with endpoints $\bm{x}$ and $\bm{y}$ is contained in $C$. The {\it convex hull} of a set $A\subset\R^3$ is the smallest convex set that contains $A$ with respect to set inclusion. For $v\geq 2$, we denote the convex hull of points $\bm{x}_1,\ldots,\bm{x}_v\in\R^3$ by $[\bm{x}_1,\ldots,\bm{x}_v]$. A {\it polyhedron} in $\R^3$ is the (possibly unbounded) intersection of a finite collection of closed halfspaces. Henceforth, we shall only consider bounded polyhedra with nonempty interior. It is well-known that every bounded polyhedron can be expressed as the convex hull of a finite set of points and vice versa (e.g., \cite{Grunbaum}). For $v\geq 3$, we say that  $P\subset\R^3$ is a  $v${\it-gon} if it is the convex hull of $v$ coplanar points  and $P$ has $v$ extreme points. 
For $v\geq 4$, we say  that $P$ is a {\it $v$-pyramid} if it is the convex hull of the union of  $(v-1)$-gon  $Q$ and an apex point $\bm{a}\not\in\affine(Q)$.  For $v\geq 5$, let $Q$ be a $(v-2)$-gon and let $I$ be a closed segment that intersects $Q$ in a single point lying in $\relint(Q)\cap\relint(I)$. The convex hull of $Q\cup I$ is called a $v${\it-bipyramid}. 
For more  background on polyhedra and convex sets, we refer the reader to, e.g., \cite{Brondsted, Grunbaum, ZieglerBook} and \cite{GruberBook, SchneiderBook}, respectively.

A {\it face} of a polyhedron (or $v$-gon) $P$ is the intersection of $P$ with a   support plane $H$ of $P$ (meaning $H$ has codimension 1, $H\cap P\neq\varnothing$ and $P\subset H^+$ or $P\subset H^-$). The faces of $P$ of dimension 0 and 1 are called {\it vertices} and {\it edges}, respectively, and the faces of $P$ of codimension 1 are called {\it facets}. 
Two polyhedra $P$ and $Q$ are {\it combinatorially equivalent} (or of the same {\it combinatorial type}) if there exists a bijection $\varphi$ between the set $\{F\}$ of all faces of $P$ and the set $\{G\}$ of all faces of $Q$ that preserves inclusions, i.e., for any two faces $F_1,F_2\in\{F\}$, the inclusion $F_1\subset F_2$ holds if and only if $\varphi(F_1)\subset\varphi(F_2)$ holds. For references on the enumeration and number of combinatorial types of polyhedra in $\R^3$ with a small number of vertices, see, e.g., \cite{BrittonDunitz1973, federico}.

     \section{Proof of Theorem \ref{mainThm}}\label{proofsection}

We will employ the method of partial variation of P\'olya \cite{Polya1954} (see also Sect. 1.4 in \cite{GeomMaxMinBook}). It says that if the function $f(x_1,\ldots,x_n)$ has a maximum (resp. minimum) at $(x_1,\ldots,x_n)=(a_1,\ldots,a_n)$, then for any $1\leq k\leq n-1$ the function $g(x_{k+1},\ldots,x_n)=f(a_1,\ldots,a_k,x_{k+1},\ldots,x_n)$ has a maximum (resp. minimum) at $(x_{k+1},\ldots,x_n)=(a_{k+1},\ldots,a_n)$. 

The proof of Theorem \ref{mainThm} requires several lemmas. The first ingredient is classical (see, e.g., \cite{DF1981} and the references therein).  
\begin{lem}\label{mainLem1}
There are precisely two combinatorial types of polyhedra with 5 vertices: 5-pyramids and 5-bipyramids.
\end{lem}
\noindent Thus, to prove Theorem \ref{mainThm} it suffices to  determine the ideal 5-pyramid and the ideal 5-bipyramid, and then  compare their surface areas. We optimize over each of the two combinatorial classes by first deriving a necessary geometric condition that  the maximizer must satisfy, which allows us to exclude most polyhedra from consideration.

\vspace{2mm}

The next ingredient is also well-known (e.g., \cite{FlorianExtremum}). 

\begin{lem}\label{k-gon maximal perim}
	Let \(k\ge 3\). Among all convex $k$-gons inscribed in a circle of radius \(R\), the regular k-gon has maximal perimeter \(2kR\sin\frac{\pi}{k}\) and maximal area \( \frac{1}{2}kR^2\sin\frac{2\pi}{k}\). 
\end{lem}


The main step in the proof of Theorem \ref{mainThm} is the following lemma, which gives a necessary condition for the maximizer.

\begin{lem}\label{fix_h:max_lateralSA}
Let $P$ be a $v$-pyramid contained in a cap of $\mathbb{B}_2^3$ of height $h$ such that the base vertices of $P$ lie in the base of the cap. Suppose also that the projection of the apex of $P$ lies in the base $(v-1)$-gon. Then the lateral surface area of $P$ is maximized if and only if $P$ has height $h$ and the base is a regular $(v-1)$-gon inscribed in the boundary of the base of the cap.
\end{lem}

\begin{proof}
First, fix $v\geq 4$ and a normal direction $\bm{u}\in\Stwo$, which will  determine the plane  that cuts off a cap of height $h$ from the unit ball; by the rotational invariance of $\Stwo$ and the surface area functional, without loss of generality we may assume that $\bm{u}=\bm{e}_3$. Also fix $h\in(0,2)$, and define the plane $H(h):=\bm{e}_3^\perp +(1-h)\bm{e}_3$. For $t\in(0,h]$, define the collection of $v$-pyramids 
\begin{align*}
\mathcal{A}_t :=\{[\bm{x}_1,\ldots,\bm{x}_v] : [&\bm{x}_1,\ldots,\bm{x}_{v-1}]\subset H(h)\cap \mathbb{B}_2^3,\, \bm{x}_v\in \intt(H(h)^+)\cap \mathbb{B}_2^3,\, \\
&\bm{x}_v^\prime\in[\bm{x}_1,\ldots,\bm{x}_{v-1}], \, \dist(\bm{x}_v,H(h))=t\}
\end{align*}
and set $\mathcal{A}_h :=\bigcup_{0< t\leq h}\mathcal{A}_t$. We aim to solve the optimization problem	\begin{equation}\label{opt:v-pyramid}
	\begin{aligned}
	\max \quad & L(P)\\
	\textrm{s.t.} \quad & P\in\mathcal{A}_h
	\end{aligned}
	\end{equation}
	where $L(P)$ denotes the lateral surface area of $P$. Let \(P^*\) be a maximizer of (\ref{opt:v-pyramid}). We show that $P^*$ has height $h$ and regular base inscribed in $H(h)\cap \Stwo$.
	
	\paragraph{Step 1.} 
	Fix \(t\in(0,h]\) and let $P\in\mathcal{A}_t$.  
	Without loss of generality, we may assume that   the sides of the base $(v-1)$-gon are $[\bm{x}_1,\bm{x}_2],\ldots,[\bm{x}_{v-2},\bm{x}_{v-1}]$ labeled inclusively in counterclockwise order. For $i=1,\ldots,v-2$, denote the corresponding side lengths by $s_i:=\|\bm{x}_i-\bm{x}_{i+1}\|_2$, and let $p_i:=\dist(\bm{x}_v^\prime,\affine([\bm{x}_i,\bm{x}_{i+1}]))$ denote the foot length from the orthogonal projection $\bm{x}_v^\prime$ of $\bm{x}_v$ onto $H(h)$ to the line containing side $i$.  With this notation,  the lateral surface area of $P$ is $L(P)=\frac{1}{2}\sum_{i=1}^{v-1} s_i \sqrt{p_i^2+t^2}$ (see Figure 2 below).
\begin{center}
\tdplotsetmaincoords{72}{26}
\def\r{1}
  \begin{tikzpicture}[scale=4.5,line join=bevel, tdplot_main_coords]
    \coordinate (O) at (0,0,0);

\coordinate (A) at (1,0,0);
\coordinate (B) at ({-1/sqrt(2)},{1/sqrt(2)},0);
\coordinate (C) at ({-1/sqrt(2)},{-1/sqrt(2)},0);
\coordinate (D) at (0,0,1);
\coordinate (E) at (0,0,{-1});
\coordinate (P) at ({-1/(sqrt(2)+1)}, {-1/8}, {2/3});
\coordinate(Pproj) at ({-1/(sqrt(2)+1)}, {-1/8}, 0);
\coordinate (F) at ({1/sqrt(2)}, {-1/sqrt(2)},0);

\begin{scope}[thick]
    \draw[dashed] (P)--(B);
    \draw (A) -- (P) -- (C);
      \draw[dashed] (A)--(B) node[midway, below]{$s_1$};

        \draw (A)--(F) node[midway,above]{$s_4$};

       \draw[dashed] (C) -- (B) node[midway,above]{$s_2$};
       
       \draw (F)--(C) node[midway,below]{$s_3$};
       
       \draw (P)--(F);

\end{scope}

\begin{scope}[opacity=0.8]
\draw[tdplot_screen_coords] (0,0,0) circle (\r);
\tdplotCsDrawLatCircle{\r}{0}
\end{scope} 

\filldraw[black] (0,0,0) circle (0.25pt) node[anchor=east] {$\bm{e}_3'$};
\filldraw[black] (0,0,1) circle (0.25pt) node[anchor=east] {$\bm{e}_3$};
\filldraw[black] (P) circle (0.25pt)  node[anchor=south] {$\bm{x}_5$};
\filldraw[black] (Pproj) circle (0.25pt) node[anchor=north] {$\bm{x}_5'$};
\filldraw[black] (A) circle (0.25pt)  node[anchor=north] {$\bm{x}_1$};
\filldraw[black] (B) circle (0.25pt)  node[anchor=south west] {$\bm{x}_2$};
\filldraw[black] (C) circle (0.25pt) node[anchor=north] {$\bm{x}_3$};
\filldraw[black] (F) circle (0.25pt) node[anchor=north] {$\bm{x}_4$};

\begin{scope}
\draw[dashed] (P)--(Pproj) node[midway,left]{$t$};
\end{scope}

\begin{scope}
\draw[thin] (D)--(O) node[midway,right]{$h$};
\end{scope}

\coordinate (S) at ({(1-7*sqrt(2))/(32+16*sqrt(2))},{(15+8*sqrt(2))/(32+16*sqrt(2))},0);
\begin{scope}
\draw[thin] (Pproj)--(S) node[midway,above]{$p_1$};
\end{scope}


\coordinate (U) at ({-1/sqrt(2)},{-1/8},0);
\begin{scope}
\draw[thin] (Pproj)--(U) node[midway,below]{$p_2$};
\end{scope}

\coordinate (V) at ({-1/(sqrt(2)+1)}, {-1/sqrt(2)},0);
\begin{scope}
\draw[thin] (Pproj)--(V) node[midway,below]{$p_3$};
\end{scope}

\coordinate (W) at ({(48-17*sqrt(2))/32}, {-(18+sqrt(2))/32},0);
\begin{scope}
\draw[thin] (Pproj)--(W) node[midway,below]{$p_4$};
\end{scope}

\tkzMarkRightAngle[size=.04](Pproj,S,A);
\tkzMarkRightAngle[size=.04](Pproj,U,B);
\tkzMarkRightAngle[size=.04](Pproj,V,C);
\tkzMarkRightAngle[size=.04](Pproj,W,A);

  \end{tikzpicture}
  
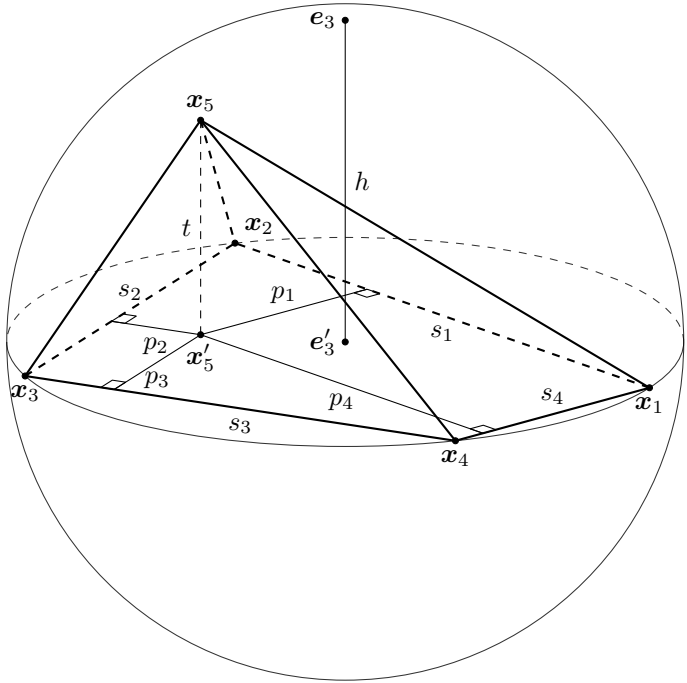
  {\footnotesize {\bf Figure 2}: The set-up for the proof of Lemma 3, with an example shown for the case $v=5$.}
\end{center}
	
		In the first step, we maximize \(L(P)=L(P,h,t,c)\) over all $v$-pyramids $P\in\mathcal{A}_t$ 
	with fixed area of the base $c>0$.  
	That is, we solve the optimization problem
	\begin{equation}\label{optv:step 0}
	\begin{aligned}
	\max \quad & L(P,h,t,c)=\frac{1}{2}\sum_{i=1}^{v-1} s_i \sqrt{p_i^2+t^2}\\
	\textrm{s.t.} \quad & P\in\mathcal{A}_t\\ 
	& \frac{1}{2}\sum_{i=1}^{v-1} s_i p_i=c.
	\end{aligned}
	\end{equation}
In order to find a necessary condition for the maximizer in  \eqref{optv:step 0}, we will instead solve a constraint-released problem. A necessary condition for the maximum surface area $v$-pyramid (resp. $v$-bipyramid) inscribed in $\Stwo$ is that the projection of the apex (apexes) lies in the interior of the base $(v-1)$-gon (central $(v-2)$-gon). Thus, to prove Theorem \ref{mainThm} it suffices to consider only those pyramids  that satisfy the condition $\bm{x}_v^\prime\in\intt([\bm{x}_1,\ldots,\bm{x}_{v-1}])$. Now, with Lemma \ref{k-gon maximal perim} in mind, we consider the  problem
	\begin{equation}\label{optv:step 1}
	\begin{aligned}
	\max \quad & L_1(P,h,t)=\frac{1}{2}\sum_{i=1}^{v-1} s_i \sqrt{p_i^2+t^2}\\
	\textrm{s.t.} \quad 
	& \frac{1}{2}\sum_{i=1}^{v-1}s_i p_i \leq \frac{1}{2}(v-1)R(h)^2\sin\frac{2\pi}{v-1}\\
	&p_i\geq 0, \,\, i=1,\ldots,v-1
	\end{aligned}
	\end{equation}
	where \(R(h):=\sqrt{1-(1-h)^2}\) is the radius of \(H(h)\cap \mathbb{B}_2^3\). The maximum is achieved on the feasible set and, since we are not restricting $s_i$ to be nonnegative,   there is no global minimum of $L_1$. The maximum is achieved when all $s_i>0$ and the problem with the constraints $s_i\geq 0$ has the same optimal solution as \eqref{optv:step 1}.   
The Lagrangian is \[\mathcal{L}:=L_1-\lambda \left(\frac{1}{2}\sum_{i=1}^{v-1}s_i p_i - \frac{1}{2}(v-1)R(h)^2\sin\frac{2\pi}{v-1}\right)+\sum_{i=1}^{v-1}\mu_i p_i
\]
where 
$\lambda,\mu_1,\ldots,\mu_{v-1}\geq 0$ are the Karush-Kuhn-Tucker (KKT) multipliers. In particular, the first-order necessary condition  $\nabla \mathcal L={\mathbf 0}$ yields that for  $i=1,\ldots,v-1$,
\begin{align*}
    \frac{\partial\mathcal{L}}{\partial s_i}&=\frac{1}{2}\sqrt{p_i^2+t^2}-\frac{\lambda p_i}{2}=0
\end{align*}
which implies
	\begin{equation}\label{inradius-vcondi}
	 p_1=\cdots=p_{v-1}.
	 \end{equation}
Thus, if \( P(h,t)\) is a maximizer of  problem (\ref{optv:step 1}), then it must satisfy  \eqref{inradius-vcondi}. Condition \eqref{inradius-vcondi}  implies that  the inball  centered at $\bm{x}_v^\prime$ with radius $r=p_1=\cdots=p_{v-1}$ is the maximal ball contained in  \([\bm{x}_1,\ldots,\bm{x}_{v-1}]\) and it is tangent to each side of the base. 
Note that a critical point of  problem \eqref{optv:step 0} is a critical point of  problem \eqref{optv:step 1}, which in turn fulfills  condition \eqref{inradius-vcondi}. Thus, the optimal polyhedron \(P^*\) satisfies  condition 
	 (\ref{inradius-vcondi}). 
	 
	\paragraph{Step 2.}  
	In step 1 we found that for each $h\in(0,2)$, among all $v$-pyramids $P\in\mathcal{A}_h$ 
	the maximum  lateral surface area is achieved at some pyramid $P^*$ satisfying \( p_1=\cdots=p_{v-1}\) and there exists a ball inscribed in and tangent to each side of its base.
	Hence,  problem \eqref{opt:v-pyramid} reduces to 
	\begin{equation}\label{optv:step 2}
	\begin{aligned}
	\max \quad & L(P,h,t)=\frac{1}{2}\left(\sum_{i=1}^{v-1}s_i\right)\sqrt{r(P,h)^2+t(P)^2}\\
	\textrm{s.t.} \quad & P\in\mathcal{A}_h\\
	& p_1=\cdots=p_{v-1}=:r(P,h).
	\end{aligned}
	\end{equation}
	Under the constraints in \eqref{optv:step 2}, the following inequalities hold: 
	\begin{itemize}
	    \item[(i)] By Lemma \ref{k-gon maximal perim}, \(\sum_{i=1}^{v-1}s_i\le 2(v-1)R(h)\sin\frac{\pi}{v-1} \)     with equality  if and only if $[\bm{x}_1,\ldots,\bm{x}_{v-1}]$ is regular and $\bm{x}_1,\ldots,\bm{x}_{v-1}\in \Stwo\cap H(h)$;
	    \item[(ii)] \( r(P,h)\le R(h)\cos\frac{\pi}{v-1}\) with equality if and only if $[\bm{x}_1,\ldots,\bm{x}_{v-1}]$ is regular \cite{Toth1948-1, Toth1948-2}; and 
	    \item[(iii)] \( t(P)\le h\) with equality if and only if $\bm{x}_v=\bm{e}_3$.
	\end{itemize}
		   Equality holds in all of (i), (ii) and (iii) simultaneously if and only if $P$ has height $h$ and the base of \(P\) is a regular $(v-1)$-gon inscribed in \(H(h)\cap \Stwo\). Therefore,  
    \begin{align}
          L(P) &\le  \frac{1}{2}\times 2(v-1)\sin\frac{\pi}{v-1}R(h)\times\sqrt{R(h)^2\cos^2\frac{\pi}{v-1}+h^2} \nonumber\\
          &=(v-1)\sin\frac{\pi}{v-1}h\sqrt{2-h}\sqrt{h\sin^2\frac{\pi}{v-1}+2\cos^2\frac{\pi}{v-1}} \label{opt:step2max},
          \end{align}
with equality   if and only if \( P\) is the right pyramid of height $h$ with regular base inscribed in $H(h)\cap\Stwo$.    
\end{proof}

Next, we use Lemma \ref{fix_h:max_lateralSA} to determine the optimal 5-pyramid inscribed in $\Stwo$. Naturally, the optimizer has a square pyramidal structure.

\begin{cor}\label{BestPyr}
Let $P$ be a 5-pyramid inscribed in $\Stwo$. Then
\[
S(P) \leq 4\eta-2\eta^2+2\sqrt{4\eta^2-\eta^4}=5.77886\ldots
\]
where $\eta:=\frac{1}{3}\left(1-\sqrt{46}\sin\left(\frac{\pi}{6}-\frac{1}{3}\arccos\left(-\frac{149}{23\sqrt{46}}\right)\right)\right)=1.2622\ldots$. Equality holds if and only if $P$ is a rotation of the 5-pyramid with height $\eta$ and square base inscribed in the circle $\Stwo\cap(\bm{e}_3^\perp-(\eta-1) \bm{e}_3)$.
\end{cor}

\begin{proof}
	We shall solve the optimization problem 
	\begin{equation}\label{optS:v-pyramid}
	\begin{aligned}
	\max \quad & S(P)\\
	\textrm{s.t.} \quad & 
	P \text{  is a }5\text{-pyramid}\\
	& P\subset \mathbb{B}_2^3.
	\end{aligned}
	\end{equation}
The maximizer has height $h\geq 1$, so we may assume this holds.  Fix $h\in[1,2)$ and let $P(h)\in\mathcal{A}_h$.   
	 By Lemma \ref{fix_h:max_lateralSA},  the lateral surface area $L(P(h))$ is maximized precisely when $P(h)$ has height $h$ and square base inscribed in $H(h)\cap\Stwo$. By Lemma \ref{k-gon maximal perim}, for any height $h$ the area of the base is also maximized when the base is a square inscribed in $\Stwo\cap H(h)$. Thus, 
	 \[
	 \begin{aligned}
	 S(P(h)) & = B(P(h))+L(P(h))  \le  4h-2h^2+2\sqrt{4h^2-h^4},
	 \end{aligned}
	 \]
	 where  \( B(P(h)):=4h-2h^2\) is the area of the square base of \(P(h)\) (here we have used that $h\geq 1$). 	 It remains to optimize over $h$. Define $F_5(h):=4h-2h^2+2\sqrt{4h^2-h^4}$. Setting 
	 \[
	 \frac{dF_5}{dh}=4-4h+\frac{8-4h^2}{\sqrt{4-h^2}}=0
	 \] 
	 we obtain the equation $2h^3-2h^2-7h+8=0$. The roots of the cubic  that lie in $[1,2)$ are
	 \begin{align*}
	     h_1 &=\frac{1}{3}\left(1 +  \sqrt{46} \cos\left(\frac{1}{3}\arccos\left(\frac{-149}{23 \sqrt{46}}\right)\right)\right)=1.6538868\ldots\\
	     h_2 &= \frac{1}{3}\left(1-\sqrt{46}\sin\left(\frac{\pi}{6}-\frac{1}{3}\arccos\left(-\frac{149}{23\sqrt{46}}\right)\right)\right)=1.2622\ldots.
	 \end{align*}
	 Checking cases, we find that $S(h)$ attains its global maximum at $h_2=:\eta$. The equality conditions follow from those in Lemma \ref{fix_h:max_lateralSA}.
\end{proof}
As an immediate corollary, it follows that the surface area discrepancy of an observed $v$-pyramid $P$ equals
\begin{equation}\label{pyrdisc}
    \mathscr{S}(P) = 1-\frac{S(P)}{F_5(\eta)}\approx 1-(0.173) S(P).
\end{equation}
In Section \ref{comparison}, we modify the previous arguments to show that the volume and surface area discrepancies are distinct for the class of $v$-pyramids with $v\geq 5$.



More generally, let $P$ be a $v$-pyramid contained in $\mathbb{B}_2^3$. 
Modifying the proof of Corollary \ref{BestPyr},  it can be shown that $S(P) \leq F_v(\eta)$, 
where 
\begin{align}\label{bestvpyrSA}
F_v(h) &:= \frac{1}{2}(v-1)(2h-h^2)\sin\frac{2\pi}{v-1} \nonumber\\
&+(v-1)\sin\frac{\pi}{v-1}\sqrt{h^2(2h-h^2)+(2h-h^2)^2\cos^2\frac{\pi}{v-1}}
\end{align}
and \( \eta\in (1,2)\) is the optimal height for which \( F(h)\le F(\eta)\) for all \(h\in (0,2)\). 
The equality holds if and only if the vertices of $P$ are (up to rotation) the north pole \(\bm{e}_3\) and the corners of the regular $(v-1)$-gon inscribed in $\Stwo\cap (\bm{e}_3^\perp-(\eta-1)\bm{e}_3)$. 
We leave the details of computing a general formula for $\eta=\eta(v)$ in terms of $v$ to the interested reader; to complete the proof of Theorem \ref{mainThm}, we will  only need the case $v=5$ from Corollary \ref{BestPyr}. In the next corollary, however, we state the formula for the surface area of the optimal $v$-bipyramid in full generality.	 


\begin{cor}\label{BPlem}
Let $P$ be a $v$-bipyramid inscribed in $\Stwo$ and let   $\omega_v:=\frac{\pi}{v-2}$. Then
\[
S(P)\leq 2(v-2)\sqrt{1+\cos^2\omega_v}\sin\omega_v
\]
with equality  if and only if $P$ is a rotation of the convex hull of  
the north and south poles $\pm \bm{e}_3$ and the  regular $(v-2)$-gon inscribed in the equator $\Stwo\cap \bm{e}_3^\perp$. 
\end{cor}
As an immediate corollary, it follows that the surface area discrepancy of an observed $v$-bipyramid $P$ equals
\begin{equation}
    \mathscr{S}(P) = 1-\frac{S(P)}{2(v-2)\sqrt{1+\cos^2\omega_v}\sin\omega_v}.
\end{equation}

\begin{proof}
Let $P$ be an arbitrary $v$-bipyramid inscribed in $\Stwo$. There exists a unique hyperplane that passes through $\intt(P)$ and contains $(v-2)$ vertices of $P$. Without loss of generality, we may assume that this hyperplane is $H(h):=\bm{e}_3^\perp+(1-h)\bm{e}_3$. There exist $P_1\in\mathcal{A}_h$ and $P_2\in\mathcal{A}_{2-h}$ such that $S(P)=L(P_1)+L(P_2)$. Thus, for each $h$, the surface area $S(P)$ is maximized if $L(P_1)$ and $L(P_2)$ are simultaneously maximized.  By Lemma \ref{fix_h:max_lateralSA}, it suffices to consider only those $v$-bipyramids $P(v,h)$ that are the convex hull of $\pm \bm{e}_3$ and a regular $(v-2)$-gon inscribed in $H(h)\cap \Stwo$. Hence, $P_1$ has height $h$, $P_2$ has height $2-h$ and their common base is a regular $(v-2)$-gon, as in Figure 3 below. 

\begin{center}
\tdplotsetmaincoords{80}{76}
\def\r{1}
  \begin{tikzpicture}[scale=3.3,line join=bevel, tdplot_main_coords]
    \coordinate (O) at (0,0,0);

\coordinate (A) at ({sqrt(15)/4},0,1/4);
\coordinate (B) at ({-sqrt(15)/8},{sqrt(45)/8},1/4);
\coordinate (C) at ({-sqrt(18)/8},{-sqrt(45)/8},1/4);
\coordinate (D) at (0,0,1);
\coordinate (E) at (0,0,{-1});
\coordinate (F) at (0,0,1/4);

\begin{scope}[thick]
    \draw (A) -- (D)--(B);
    \draw (A) -- (D) -- (C);
    \draw (A) -- (B)--(E);
    \draw (A)--(C)--(E);
    \draw (A)--(E)--(B);
\end{scope}

\draw[thick,fill=blue, opacity=0.1] (A) -- (D)--(B);
\draw[thick,fill=blue, opacity=0.1] (A) -- (D) -- (C);
\draw[thick,fill=blue,opacity=0.1](A) -- (C) -- (E);  
\draw[thick,fill=blue,opacity=0.1] (A)--(E)--(B);  

\begin{scope}[dashed, thick] 
    \draw (C) -- (B);
\end{scope}

\begin{scope}
    \draw[thin] (D) -- (F);
\end{scope}

\draw [decorate,decoration={brace,amplitude=10pt},xshift=-4pt,yshift=0pt]
(F) -- (D) node [black,midway,xshift=-0.6cm] 
{$h$};

\begin{scope}
    \draw[thin] (E) -- (F);
\end{scope}

\draw [decorate,decoration={brace,amplitude=10pt},xshift=-4pt,yshift=0pt]
(E) -- (F) node [black,pos=0.5,xshift=-0.9cm] 
{$2-h$};

\begin{scope}[opacity=0.8]
\draw[tdplot_screen_coords] (0,0,0) circle (\r);
\end{scope} 

\begin{scope}[opacity=0.6]
\tdplotCsDrawLatCircle{\r}{14.477512185929923878771034799127166005131597624556616476050118008}
\end{scope} 

\filldraw[black] (0,0,1/4) circle (0.25pt) node[anchor=west] {$\mathbf{0}_h$};
\filldraw[black] (0,0,1) circle (0.25pt) node[anchor=south] {$\bm{e}_3$};
\filldraw[black] (0,0,-1) circle (0.25pt) node[anchor=north] {$-\bm{e}_3$};
\filldraw[black] ({sqrt(14)/4},0,1/4) circle (0.25pt) node[anchor=north west] {$\bm{\zeta}^h_1$};
\filldraw[black] (B) circle (0.25pt) node[anchor=south] {$\bm{\zeta}^h_2$};
\filldraw[black] (C) circle (0.25pt) node[anchor=east] {$\bm{\zeta}^h_{v-2}$};

  \end{tikzpicture}

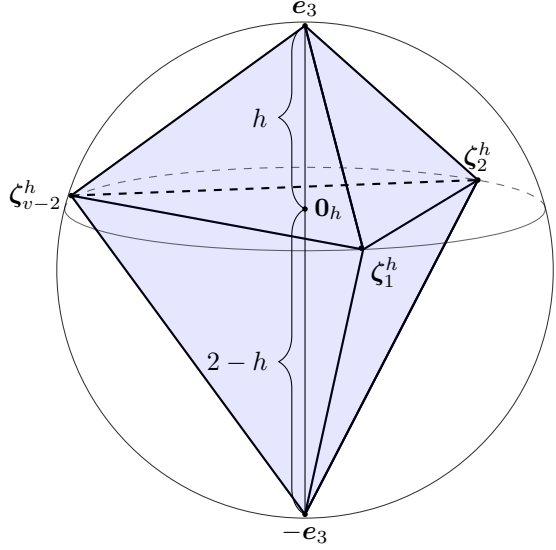
  {\footnotesize {\bf Figure 3}: The set-up for the proof of Corollary 2 in the case $v=5$. For $v\geq 5$, the points $\bm{\zeta}_i^h$, $1\leq i\leq v-2$, are the $(v-2)$th roots of unity in the circle $\mathbb{S}^2\cap (\bm{e}_3^\perp +(1-h)\bm{e}_3)$ with center $\mathbf{0}_h:=(1-h)\bm{e}_3$.}
  \end{center}

It remains to optimize over $h$. The surface area $S(h)$ of such a $v$-bipyramid  is
\begin{align*}
S(h) =(&v-2)\sin\omega_v \\
&\times\left(\sqrt{2h^3-h^4+(2h-h^2)^2\cos^2\omega_v}+\sqrt{(2-h)^2(2h-h^2)+(2h-h^2)^2\cos^2\omega_v}\right).
\end{align*}
A short computation yields that $h=1$ is the only critical point, and that $S^\prime(h)>0$ for $h\in(0,1)$ and $S^\prime(h)<0$ for $h\in(1,2)$. Thus, the vertices of the regular $(v-2)$-gon lie in the equator $\Stwo\cap \bm{e}_3^\perp$. This shows that  the polyhedron  defined in the statement of the lemma maximizes surface area among all $v$-bipyramids,  and it has surface area 
\[
S(1) = 2(v-2)\sqrt{1+\cos^2\omega_v}\sin\omega_v.
\]
\end{proof}


\subsection{Conclusion of the Proof of Theorem \ref{mainThm}}
By Corollary \ref{BestPyr}, the maximum surface area of a 5-pyramid inscribed in $\Stwo$ is less than 5.78. By Corollary \ref{BPlem}, the maximum surface area 5-bipyramid inscribed in $\Stwo$ is a rotation of $[\bm{e}_3,-\bm{e}_3,  \bm{e}_1,\bm{\zeta}_1,\bm{\zeta}_2]$ with surface area $3\sqrt{15}/2>5.78$. Thus, by Lemma \ref{mainLem1}, the triangular bipyramid $Q_5^*=[\bm{e}_3,-\bm{e}_3,  \bm{e}_1,\bm{\zeta}_1,\bm{\zeta}_2]$ maximizes surface area among all polyhedra with 5 vertices that are contained in $\mathbb{B}_2^3$.   \qed

\vspace{3mm}

We conclude this section with a couple of  observations related to the proofs.

\begin{remark}
A necessary condition for an optimal 5-pyramid (resp. 5-bipyramid) is that the orthogonal projection of the apex lies  (apexes lie)  in the  base (triangular cross-section containing three vertices). Under this condition, one can express the surface area of the base (triangular cross-section containing three vertices) in two ways to get the constraint
\begin{align*}
g &:=\frac{1}{2}\sum_{i=1}^3 s_i p_i -\frac{1}{4}\sqrt{(s_1^2+s_2^2+s_3^2)^2-2(s_1^4+s_2^4 +s_3^4) }=0\\
(\text{resp. }g &:=\frac{1}{2}\sum_{i=1}^4 s_i p_i -\frac{1}{4}\sqrt{(s_1^2+s_2^2+s_3^2 +s_4^2)^2+8s_1 s_2 s_3 s_4-2(s_1^4+s_2^4+s_3^4 +s_4^4) }=0)
\end{align*}
where we have used Heron's formula (Bragmagupta's formula). Using Lagrange multipliers, one maximizes the lateral surface area $L(P)=\frac{1}{2}\sum_i s_i\sqrt{p_i^2+t^2}$ under this constraint to derive that all $p_i$ are equal and all $s_i$ are equal. A generalization of Heron's formula was given in \cite{MRR}, and explicit formulas were proven for $v$-gons with up to 8 vertices. These formulas can be used in step 1 to conclude the  result for all $v$-pyramids with $v\leq 9$.
\end{remark}

\begin{remark}
One can use Lemma \ref{fix_h:max_lateralSA} to determine the tetrahedron of maximum surface area inscribed in $\Stwo$, which is already known \cite{Heppes, Krammer}; see also \cite{tanner, Toth}. Minor modifications can be made to the proof of Theorem \ref{mainThm} to determine the maximum volume polyhedron with five vertices inscribed in $\Stwo$, which is also already known  \cite{BermanHanes1970}.
\end{remark}

\section{Comparison of the Volume and Surface Area Discrepancies for Pyramids and Bipyramids}\label{comparison}


 Modifying the proof of Lemma \ref{fix_h:max_lateralSA}, it can be shown that among all $v$-pyramids in $\mathcal{A}_h$, the maximum volume is achieved precisely when the base is regular and the pyramid has height $h$. Optimizing over $h\in[1,2)$ as in Corollary \ref{BestPyr}, we derive that  the maximum volume $v$-pyramid in $\mathcal{A}_h$ ($h\geq 1$) has regular base and height $h$, and its volume equals
\begin{equation}\label{bestvpyrvol}
    V(h) := \frac{v-1}{6}\left(2h^2-h^3\right)\sin\frac{2\pi}{v-1}.
\end{equation}
The only critical point in $[1,2)$ is $h=\tfrac{4}{3}$, which yields the unique maximum volume $v$-pyramid for any $v\geq 4$.  On the other hand, in \eqref{bestvpyrSA} we gave the maximum surface area $F_v(h)$ that can be achieved by an inscribed $v$-pyramid. It is an elementary computation to show that  $F_v^\prime(\tfrac{4}{3})\neq 0$ for any $v\geq 5$. Thus for any $v\geq 5$, the volume and surface area discrepancies   are distinct measures of distortion on the class of $v$-pyramids.  

For the class of $v$-bipyramids,  modifications to the previous arguments show that the volume and surface area maximizers coincide (up to rotations). Hence, the volume and surface area discrepancies are equivalent on the class of $v$-bipyramids. The maximum volume achieved by the ideal $v$-bipyramid is $\tfrac{1}{3}(v-2)\sin\tfrac{2\pi}{v-2}$, and thus the volume discrepancy of an observed $v$-bipyramid $P$ equals
\begin{equation}
    \mathscr{V}(P) = 1-\left(\frac{3\csc\tfrac{2\pi}{v-2}}{v-2}\right)V(P).
\end{equation}


\section{Summary and Discussion}\label{Problems}

In \eqref{SAdiscrepancy} we defined the surface area discrepancy $\mathscr{S}(P)$ between a coordination polyhedron $P$  and the (combinatorially equivalent) ideal polyhedron $\widehat{Q}_{\text{SA}}(P)$  that maximizes surface area. In Corollaries \ref{BestPyr} and \ref{BPlem} we analytically determined the $v$-pyramid and $v$-bipyramid, respectively, of maximum surface area inscribed in $\Stwo$. Our proofs show that the maximizers are unique up to rotations. The resulting formulas   \eqref{SAdiscform} and \eqref{pyrdisc} can be used in applications to compute the surface area discrepancy  of any observed coordination polyhedron with five ligand atoms. In Section \ref{comparison} we showed  that the volume and surface area discrepancies are not equivalent for certain types of polyhedra, such as the $v$-pyramids. We used the corollaries to prove the main result, which states that among all polyhedra with $v=5$ vertices inscribed in the sphere $\mathbb{S}^2$, the global surface area maximizer $Q_5^*$ is the triangular bipyramid $[\bm{e}_3, -\bm{e}_3, \bm{e}_1, \bm{\zeta}_1, \bm{\zeta}_2]$ with surface area $S(Q_5^*)=3\sqrt{15}/2$. Our proof shows that the maximizer is unique up to rotations. 

The cases $v\leq 12$ for which the global surface area maximizer $Q_v^*$ has been determined explicitly are listed in Table 1. Prior to this work, the maximum surface area polyhedron with $v$ vertices inscribed in $\Stwo$ was determined for the cases $v=4,6,12$. In the case $v=5$, Akkiraju \cite{akkiraju} defined a local optimality condition, but did not give a  proof which explicitly determined the maximizer. Numerical simulations led Akkiraju to conjecture that Theorem \ref{mainThm} holds, and he asked for a  proof of this result. Theorem \ref{mainThm} confirms Akkiraju's conjecture in the affirmative. 

To the best of our knowledge, it is an open problem to determine the maximum surface area polyhedron $Q_v^*$  inscribed  in $\Stwo$ with $v$ vertices for $7\leq v\leq 11$ and $v\geq 13$. Based on numerical investigations, in the case $v=7$ we conjecture that $Q_7^*$  is the convex hull of the north and south poles and five vertices forming an equilateral pentagon in the equator (called a {\it pentagonal bipyramid}). Already in the case $v=7$, there are 34 combinatorial types of polyhedra (e.g., \cite{federico}), and the number of types explodes as $v$ increases (e.g., \cite{Grunbaum}). Thus, one needs some strong necessary condition(s) that can be used to eliminate most combinatorial types from consideration. 
More generally, for $d\geq 4$ it is an open problem to determine the polyhedron of maximum surface area inscribed in $\mathbb{S}^{d-1}$ with $v\geq d+2$ vertices.

The work \cite{akkiraju} was motivated in part by that of  \cite{BermanHanes1970}, where the maximum volume polyhedra inscribed in $\Stwo$ with $v\leq 8$ vertices were determined analytically. 
The later work \cite{HorvathLangi} extended the methods in \cite{BermanHanes1970} to  determine the maximum volume polyhedron inscribed in $\Sp$ with $d+2$ vertices, $d\geq 2$, and also with $d+3$ vertices when $d$ is odd. A key step in the proof is showing that the maximizer is simplicial, meaning each facet is a simplex in its affine hull. Perhaps those arguments  can be adapted to find the maximum surface area polyhedron inscribed in $\mathbb{S}^{d-1}$ with $d+2$ vertices.

The global volume maximizer  coincides with the global surface area maximizer   in all of the  cases  where both optimizers are known explicitly, which are $v=4,5,6,12$. It is natural to ask if this holds (up to rotations) for all $v\geq 4$. This was conjectured in \cite{Kazakov} in the context of quantum theory. If true, this result would show that the surface area discrepancy is equivalent to the volume discrepancy when the observed polyhedron has the same combinatorial type as the  global maximizer. 

\section*{ACKNOWLEDGMENTS}

The authors express their sincerest gratitude to Shiri Artstein-Avidan and Florian Besau for their remarks, and to the anonymous referees for carefully reading our paper and providing  helpful comments and constructive feedback. The first two named authors thank the Perspectives on Research In Science \& Mathematics (PRISM) program at Longwood University for its  support. The third named author has received funding from the European Research Council (ERC) under the European Union’s Horizon 2020 research and innovation programme (grant agreement No 770127). 

	\bibliographystyle{acm}
\bibliography{bibloDHL-final}

	\vspace{3mm}

\noindent {\sc Department of Mathematics \& Computer Science, Longwood University, 23909}

\noindent {\it E-mail address:} {\tt jessica.donahue@live.longwood.edu}

\vspace{2mm}

\noindent {\sc Department of Mathematics \& Computer Science, Longwood University, 23909}

\noindent {\it E-mail address:} {\tt hoehnersd@longwood.edu}

\vspace{2mm}

\noindent {\sc School of Mathematics and Statistics, Ningbo University, Ningbo 315211, P.R. China}

\noindent {\it E-mail address:} {\tt bxl292@case.edu}

\end{document}